\documentclass{svproc}

\usepackage{amsfonts} 
\usepackage{amsmath}
\usepackage{amsopn}
\usepackage{amssymb}
\usepackage{booktabs}
\usepackage{caption}
\usepackage{commath}
\usepackage{csvsimple}
\usepackage{enumitem}
\usepackage[T1]{fontenc}
\usepackage{hyperref}
\usepackage{mathtools}
\usepackage{subcaption}
\usepackage{tikz}
\usepackage{thmtools}
\usepackage{slantsc}
\usepackage{nccmath}
\usepackage{svg}
\usepackage{cleveref}
\usepackage[utf8]{inputenc}

\usepackage{pgfplots}
\usepackage{pgfplotstable}
\pgfplotsset{compat=1.18}
\usepgfplotslibrary{groupplots}

\definecolor{ferngreen}{HTML}{789024}
\definecolor{winered}{HTML}{A63852}
\definecolor{perfumepurple}{HTML}{c0affb}
\definecolor{apricotorange}{HTML}{e6a176}
\definecolor{orientblue}{HTML}{00678a}
\definecolor{downygreen}{HTML}{5eccab}

\definecolor{cq10}{HTML}{000000}
\definecolor{cq11}{HTML}{0072B2}
\definecolor{cq13}{HTML}{009E73}
\definecolor{cq15}{HTML}{D55E00}

\allowdisplaybreaks

\newcommand{\getlastpoint}[1]{%
  \pgfplotstablegetrowsof{#1}%
  \pgfmathtruncatemacro{\lastrow}{\pgfplotsretval-1}%
  \pgfplotstablegetelem{\lastrow}{cost}\of{#1}%
  \edef\lastX{\pgfplotsretval}%
  \pgfplotstablegetelem{\lastrow}{RMSE}\of{#1}%
  \edef\lastY{\pgfplotsretval}%
}

\newcommand{\setstartX}{%
  \begingroup
  \pgfkeys{/pgf/fpu=true, /pgf/fpu/output format=fixed}%
  \pgfmathsetmacro{\startX}{2.5e6}%
  \pgfmathsmuggle\startX
  \endgroup
}

\newtheorem{assumption}{\upshape Assumption}
\crefname{assumption}{assumption}{assumptions}
\crefname{property}{property}{properties}

\DeclarePairedDelimiter\normZZZZ{\lVert}{\rVert}
\DeclarePairedDelimiter\absZZZZ{\lvert}{\rvert}

\DeclareSymbolFont{bbold}{U}{bbold}{m}{n}
\DeclareSymbolFontAlphabet{\mathbbold}{bbold}

\newcommand{\floor}[1]{\left\lfloor #1 \right\rfloor} 

\newenvironment{proofof}[1]{%
  \renewcommand{\proofname}{#1}%
  \begin{proof}%
}{%
  \end{proof}%
}

\begin{document}


\titlerunning{Single-ensemble multilevel McKean--Vlasov}
\authorrunning{Arne Bouillon \and Giovanni Samaey}
\tocauthor{Arne Bouillon and Giovanni Samaey}
\title{Convergence of a single-ensemble multilevel scheme for McKean--Vlasov SDEs}
\author{Arne Bouillon \and Giovanni Samaey}
\institute{NUMA research group, Department of Computer Science, KU Leuven, Leuven, Belgium (\href{mailto:arne.bouillon@kuleuven.be}{arne.bouillon@kuleuven.be}, \href{mailto:giovanni.samaey@kuleuven.be}{giovanni.samaey@kuleuven.be})}
\maketitle

\begin{abstract}
    Numerically solving McKean--Vlasov stochastic differential equations is computationally challenging due to the compounding costs of discretizing in time and in the distribution of the solution. Multilevel ideas have been proposed to provide speed-ups. In this work, we study the multilevel Monte Carlo method proposed by Ricketson (2015) for equations whose drift and diffusion terms depend on the law of the solution~$X_t$ through the expectation~$\mathbb E[R(X_t)]$. The scheme follows the \emph{single-ensemble} paradigm, where particles interact across levels at each timestep. While cross-level feedback makes this scheme attractive in practice, the correlations it introduces have so far confined its cost-error analysis to a model problem with linear drift, deterministic diffusion, and~$R$ the identity. We use additional coarse particles to enforce geometrically decaying coupling errors towards the coarser levels. This allows us to prove our main contribution, an~$L^p$-error of~$\mathcal O(\epsilon)$ at cost~$\mathcal O(\epsilon^{-2-\delta})$ for any~$p\ge2$ and~$\delta>0$ (with a constant that grows as~$\delta\to0$), assuming only global Lipschitz bounds on the drift, the diffusion, and~$R$. An exploratory experiment is consistent with the derived rates and suggests that in practice the constant does not grow significantly for small $\delta$. Our methodology and proof strategy may also be useful for other single-ensemble multilevel schemes, such as multilevel ensemble~Kalman~filters.

    \keywords{McKean--Vlasov SDEs, multilevel Monte Carlo, single-ensemble}
\end{abstract}

\section{Introduction}
    We consider the stochastic differential equation (SDE)
    \begin{equation} \label{eq:mckv}
        \mathrm dX_t = b(X_t, \mathbb E[R(X_t)])\,\mathrm dt + \sigma(X_t, \mathbb E[R(X_t)])\,\mathrm dW_t, \qquad X_0 \sim \rho_0,
    \end{equation}
    where $X_t\in\mathbb R^d$ represents a state variable at time $t$, $\rho_0$ specifies the initial distribution, $R\colon\mathbb R^d\to\mathbb R^{d_R}$ is a function, and $b$ and $\sigma$ represent the SDE's drift and diffusion terms, respectively. This equation is of McKean--Vlasov type, as its evolution depends on the law of its state and not only on the state itself. McKean--Vlasov equations arise in various applications, including the \mbox{social~sciences}~\cite{carmonaProbabilisticApproachMean2016,tembineMeanFieldDifference2011}, biology \cite{talayNewMcKeanVlasov2020,zhuHybridRisksensitiveMeanfield2011}, and plasma physics \cite{fournierPropagationChaosLandau2016}; specifically, \cref{eq:mckv} appears in models for muscle contraction \cite{frankNonlinearFokkerPlanckEquations2005,shimizuPhenomenologicalEquationsMotion1972} and finance \cite{loSimpleAnalyticalModel2012}, and in particle-in-cell methods for the Vlasov--Poisson system used in plasma physics \cite{vayMeshRefinementParticleincell2002} -- see \cite[section 5.3]{ricketsonMultilevelMonteCarlo2015}.

    Throughout the paper, we consider $b$ and $\sigma$ independent of the time $t$, but we note that our results extend to time-dependent coefficients.

    Numerical simulation of \cref{eq:mckv} must address two sources of intractability: not only must time be discretized (as in regular SDEs), the expectation $\mathbb E[R(X_t)]$ must also be approximated. This can be done by simulating an \emph{ensemble} of $J$ paths (\emph{particles}) $\mathbf x_t \coloneqq \{X_t^j\}_j$ of \cref{eq:mckv} in parallel and using their empirical distribution instead of the true distribution of $X_t$. In this way, all particles interact at each time step. That is, we approximate
    \begin{equation}
        \mathbb E[R(X_t)] \approx \frac1{|\mathbf x_t|}\sum\nolimits_{X_t^j\in\mathbf x_t} R(X_t^j) \eqqcolon E_R(\mathbf x_t)
    \end{equation}
    and after an Euler--Maruyama time discretization with timestep $\Delta t$ arrive at
    \begin{equation} \label{eq:mc}
        X_{n+1}^j = X_n^j + b(X_n^j, E_R(\mathbf x_n))\Delta t + \sigma(X_n^j, E_R(\mathbf x_n)) \Delta W_n^j,
    \end{equation}
    a Monte Carlo scheme. The need to discretize in both time and distribution causes the computational cost to rise sharply as the required accuracy in approximating \cref{eq:mckv} tightens: to improve the error, we must simulate more particles that \emph{each} take more (smaller) timesteps. These costs multiply.

    In this work, we show that a \emph{multilevel Monte Carlo} scheme -- which simulates particles with different timesteps, all interacting together, to reduce computational cost -- can be analyzed rigorously under global Lipschitz assumptions. We first discuss this and related schemes in \cref{sec:intro:mlmc} and then review our contributions in detail in \cref{sec:intro:obj}.

    \subsection{Multilevel Monte Carlo} \label{sec:intro:mlmc}
    Various multilevel Monte Carlo (MLMC) \cite{giles2008,heinrichMultilevelMonteCarlo2001a} techniques have been proposed to leverage the efficiency of cheap, coarse simulations while still achieving the accuracy of expensive, fine simulations. To introduce MLMC, let $Q=Q_\infty$ be an intractable random variable and let $\{Q_\ell\}_{\ell=0}^L$ be a sequence of random variables where, as the \emph{level} $\ell$ increases, the approximation becomes more accurate but~$Q_\ell$ becomes more expensive to simulate. MLMC estimates the expected value of~$Q$ by the telescopic sum
    \begin{equation} \label{eq:mlmc}
    \begin{aligned}
        {\mathbb{E}}[Q] \approx {\mathbb{E}}[Q_L] &= {\mathbb{E}}[Q_0] + \sum\nolimits_{\ell=1}^L {\mathbb{E}}[Q_\ell - Q_{\ell-1}]\\
        &\approx \frac1{J_0}\sum\nolimits_{j=1}^{J_0}Q_0^j + \sum\nolimits_{\ell=1}^L\frac1{J_\ell}\sum\nolimits_{j=1}^{J_\ell}(Q_\ell^j - \widetilde Q_{\ell}^j),
    \end{aligned}
    \end{equation}
    where $Q_\ell^j$ and $\widetilde Q_{\ell}^j$ are coupled samples from $Q_\ell$ and $Q_{\ell-1}$, respectively. Good coupling makes the difference $Q_\ell^j - \widetilde Q_\ell^j$ a low-variance estimator, such that $J_\ell$ can decrease quickly with $\ell$. This approach often results in asymptotically lower costs for approximating ${\mathbb{E}}[Q]$ to a given accuracy than simulating $Q_L$ directly.
    
    MLMC methods for McKean--Vlasov equations combine many samples at coarse levels with fewer at fine levels. A particle's level determines its timestep and hence its cost. Current algorithms roughly divide into three~\mbox{categories}.
    
    First, what we will call \emph{single-ensemble} methods use one mixed-level ensemble, with particles at different levels interacting throughout the simulation with estimations like \cref{eq:mlmc}. This paradigm includes \cite{ricketsonMultilevelMonteCarlo2015}, which considers equations of the exact form of \cref{eq:mckv}; a cost-error analysis is provided for the special case
    \begin{equation}
        \mathrm dX_t = (AX_t + B{\mathbb{E}}[X_t])\,\mathrm dt + \sigma(t)\,\mathrm dW_t,
    \end{equation}
    finding that ${\mathbb{E}}[\varphi(X_T)]$ for Lipschitz $\varphi$ can be approximated to $L^1$-error $\epsilon$ at cost $\mathcal O(\epsilon^{-2}\absZZZZ{\log\epsilon}^5)$, as opposed to $\mathcal O(\epsilon^{-3})$ in the single-level case. Similarly in the single-ensemble family are a group of methods for discrete-time processes (i.e., with fixed timestep $\Delta t$) such as the ensemble Kalman filter that vary drift and diffusion accuracy instead of $\Delta t$ across levels; see, e.g., \cite{bouillonSingleEnsembleMultilevelMonte2026a,chernov2021,hoelMultilevelEnsembleKalman2016}. These methods have cost-error bounds outside the linear regime, but with an error that is $\mathcal O(\epsilon\absZZZZ{\log\epsilon}^N)$, where $N$ is the number of timesteps. While this logarithmic factor does not manifest in experiments, it prevents a similar cost-error analysis for continuous-time McKean--Vlasov SDEs. In general, single-ensemble methods are seen as very challenging to analyze due to the correlations that are introduced by the cross-level interactions \cite{hoelMultiindexEnsembleKalman2022b,szpruchIterativeMultilevelParticle2019c}.

    Second, \emph{multiple-ensemble} algorithms simulate a large number of single-level ensembles like \cref{eq:mc}, most at coarse levels, coupled through shared randomness. The telescoping sum in \cref{eq:mlmc} is only applied at the end, to approximate a quantity of interest over $\mathrm{Law}(X_t)$. Such methods include \cite{baoMilsteinSchemesAntithetic2024,botija-munozMultilevelMonteCarlo2023,haji-aliMultilevelMultiindexMonte2018,hoelMultilevelEnsembleKalman,szpruchAntitheticMultilevelSampling2021a}. This approach makes analysis more viable, as the cross-level correlations are removed, and also facilitates extensions such as antithetic \cite{baoMilsteinSchemesAntithetic2024,benrachedMultilevelImportanceSampling2024a,szpruchAntitheticMultilevelSampling2021a} and multi-index \cite{haji-aliMultilevelMultiindexMonte2018,hoelMultiindexEnsembleKalman2022b} MLMC, and rare-event sampling \cite{benrachedMultilevelImportanceSampling2024a}. Various $\mathcal O(\epsilon^{-2})$ or near-$\mathcal O(\epsilon^{-2})$ cost-error results have been obtained for these algorithms. A numerical experiment in \cite{hoelMultilevelEnsembleKalman} compares single- and multiple-ensemble methods and finds similar performance, with a slight advantage for the single-ensemble algorithm.

    Third, iterative algorithms treat the interaction term as fixed so that they can simulate classical independent MLMC paths. Based on a multilevel estimator, the interaction term is then updated and the process is repeated until convergence. Examples in this class include \cite{belomestnyIterativeMultilevelDensity2019b,hutzenthalerMultilevelPicardApproximations2022,neufeldMultilevelPicardApproximations2026,szpruchIterativeMultilevelParticle2019c}.

    We note that the state of the art in the latter two families is not directly comparable in scope to \cite{ricketsonMultilevelMonteCarlo2015}'s single-ensemble method and to the present work. They typically target more general McKean--Vlasov formulations, with coefficients that may depend on $\mathrm{Law}(X_t)$ in more general ways than through the expectation $\mathbb E[R(X_t)]$. Some complexity bounds were obtained under correspondingly stronger assumptions. For instance, the results in \cite{baoMilsteinSchemesAntithetic2024,haji-aliMultilevelMultiindexMonte2018} require variance reduction rates that, to our knowledge, have not yet been proven to hold~in~\mbox{general}.\hspace{-1cm}

    \subsection{Objectives and contributions} \label{sec:intro:obj}
    The aim of this work is to show that, with only two minor algorithmic changes and with an arbitrarily small penalty to the asymptotic cost, the single-ensemble method from \cite{ricketsonMultilevelMonteCarlo2015} is amenable to a rigorous cost-error analysis.
    
    By introducing additional coarse particles to force the coupling errors between levels to decay geometrically, the multilevel estimation error can be controlled; after that, the remaining analysis is largely classical. We show that, for ${\mathbb{E}}[\absZZZZ{X_0}^p]<\infty$ and globally Lipschitz $b$, $\sigma$, $R$, and $\varphi$, a quantity of interest ${\mathbb{E}}[\varphi(X_T)]$ of \cref{eq:mckv} at fixed time $T$ can be approximated with MLMC to $L^p$-error $\mathcal O(\epsilon)$ at cost $\mathcal O(\epsilon^{-2-\delta})$ for any $p\ge2$ and $\delta > 0$. The constant in the error bound grows as $\delta\to0$, but an experiment finds that empirical errors follow the predicted rates without this penalty. This adapted scheme and its convergence analysis are presented in \cref{sec:scheme,sec:analysis}, respectively. We then provide an illustrative numerical experiment in \cref{sec:num}, followed by a discussion of our results in \cref{sec:disc}.

    With our work, we hope to strengthen the foundation of single-ensemble methods and reduce the degree to which their analysis is an obstacle to further development. While we purposefully use a simple problem class here (interaction term ${\mathbb{E}}[R(X_t)]$, Lipschitz assumptions), we believe that our methodology may apply to more general settings such as multilevel ensemble~Kalman~\mbox{filters}~\cite{bouillonSingleEnsembleMultilevelMonte2026a,chernov2021,hoelMultilevelEnsembleKalman2016}.

\section{The multilevel scheme} \label{sec:scheme}
    The algorithm we consider is a slight modification of the multilevel scheme from \cite{ricketsonMultilevelMonteCarlo2015}. That method assigns particles to various levels. Level-0 particles $\mathbf x_n^0 \coloneqq \{X_n^{0,j}\}_{j=1}^{J_0}$ follow \cref{eq:mc} for a large timestep $\Delta t_0$:
    \begin{subequations} \label{eq:scheme}
    \begin{equation}
        X_{n+1}^{0,j} = X_n^{0,j} + b(X_n^{0,j}, E_R(\mathbf x_n^0))\Delta t_0 + \sigma(X_n^{0,j}, E_R(\mathbf x_n^0)) \Delta W_n^{0,j}.
    \end{equation}
    To correct these inaccurate particles, pairs of \emph{fine} particles $\mathbf x_n^\ell \coloneqq \{X_n^{\ell,j}\}_{j=1}^{J_\ell}$ and \emph{coarse} particles $\mathbf{\widetilde x}_n^\ell \coloneqq \{\widetilde X_n^{\ell,j}\}_{j=1}^{J_\ell}$ on levels $\ell\in\{1, \ldots, L\}$ are coupled through shared Brownian motion; coupled particles differ only through their timesteps ($\Delta t_\ell$ and $\Delta t_{\ell-1}$) and interaction terms ($E_R^{\mathrm{ML}}(\mathbf x_n^{{\mathrm{ML}},\ell})$ and $E_R^{\mathrm{ML}}(\mathbf x_n^{{\mathrm{ML}},\ell-1})$):
    \begin{equation}
    \begin{aligned}
        X_{n+1}^{\ell,j} &= X_n^{\ell,j} + b(X_n^{\ell,j}, E_R^{\mathrm{ML}}(\mathbf x_n^{{\mathrm{ML}},\ell}))\Delta t_\ell \\
        &\quad + \sigma(X_n^{\ell,j}, E_R^{\mathrm{ML}}(\mathbf x_n^{{\mathrm{ML}},\ell})) \Delta W_n^{\ell,j}, && 1\le\ell\le L,\\[1ex]
        \widetilde X_{n+1}^{\ell,j} &= \widetilde X_n^{\ell,j} + b(\widetilde X_n^{\ell,j}, E_R^{\mathrm{ML}}(\mathbf x_n^{{\mathrm{ML}},\ell-1}))\Delta t_{\ell-1} \\
        &\quad + \sigma(\widetilde X_n^{\ell,j}, E_R^{\mathrm{ML}}(\mathbf x_n^{{\mathrm{ML}},\ell-1})) \Delta \widetilde W_n^{\ell,j}, && 1\le\ell\le L.
    \end{aligned}
    \end{equation}
    \end{subequations}
    Consider a domain $n\Delta t_\ell\in[0, T]$. In~\cref{eq:scheme},
    \begin{equation}
        \Delta W_n^{\ell,j} = W_{(n+1)\Delta t_\ell}^{\ell,j} - W_{n\Delta t_\ell}^{\ell,j} \qquad \text{and} \qquad \Delta \widetilde W_n^{\ell,j} = W_{(n+1)\Delta t_{\ell-1}}^{\ell,j} - W_{n\Delta t_{\ell-1}}^{\ell,j}
    \end{equation}
    for some underlying Brownian motion $W^{\ell,j}$. The paths $X^{\ell,j}$ and $\widetilde X^{\ell,j}$ are coupled and use $X_0^{\ell,j} = \widetilde X_0^{\ell,j}\sim\rho_0$ for some given $\rho_0$, but their time points do not coincide. We choose timesteps $\Delta t_{\ell-1} = a\Delta t_\ell$ for some $a\in\mathbb R$ with $a>1$ (generalizing the $\Delta t_{\ell-1}=2\Delta t_\ell$ that is assumed in \cite{ricketsonMultilevelMonteCarlo2015}).

    Two aspects are as yet undefined, and they are the areas where we deviate algorithmically from \cite{ricketsonMultilevelMonteCarlo2015}. Firstly, we define multilevel ensembles as
    \begin{equation}
        \mathbf x_n^{{\mathrm{ML}},\ell} \coloneqq \left\{\mathbf x_{\floor{n/{a^\ell}}}^0, (\mathbf x_{\floor{n/{a^{\ell-1}}}}^1, \mathbf{\widetilde x}_{\floor{n/{a^\ell}}}^1), \ldots, (\mathbf x_n^\ell, \mathbf{\widetilde x}_{\floor{\frac na}}^\ell)\right\}
    \end{equation}
    and corresponding multilevel estimates of the interaction term as
    \begin{equation} \label{eq:ml-int}
        E_R^{\mathrm{ML}}(\mathbf x_n^{{\mathrm{ML}},\ell}) = E_R(\mathbf x_{\floor{n/{a^\ell}}}^0) + \sum\nolimits_{\ell'=1}^\ell\biggl(E_R\Bigl(\mathbf x_{\floor{n/{a^{\ell-\ell'}}}}^{\ell'}\Bigr) - E_R\Bigl(\mathbf{\widetilde x}_{\floor{n/{a^{\ell-\ell'+1}}}}^{\ell'}\Bigr)\biggr).
    \end{equation}
    In other words, a particle at level $\ell$ interacts only with particles at levels $\ell$ and below, and particles' paths are seen as piecewise constant by higher-level particles. In \cite{ricketsonMultilevelMonteCarlo2015}, lower-level paths were interpolated instead of frozen at their last value. Secondly, our specific choice of $J_\ell$ -- using more coarse particles than the original algorithm -- will be detailed in \cref{thm:main} and is crucial to our analysis.

    Taking a step back, the goal of simulating McKean--Vlasov equations is often to approximate a quantity of interest (QoI) ${\mathbb{E}}[\varphi(X_T)]$ for some function $\varphi$. The multilevel QoI estimator is $E_\varphi^{\mathrm{ML}}(\mathbf x_N^{{\mathrm{ML}},L})$ with $N = \floor{T/\Delta t_L}$, defined analogously to the multilevel interaction estimator. By combining many coarse-level particles with fewer fine-level particles, the goal is that this estimator achieves the same accuracy as a single-level Monte Carlo estimator at level $L$ but at a lower cost.
    
    The algorithm can be implemented by iterating over levels $\ell=0, \ldots, L$ and, for each level, simulating the entire time interval $[0, T]$ using stored lower-level interaction contributions. Evolving all levels simultaneously is slightly more complicated but can be more memory-efficient.

\section{Analysis} \label{sec:analysis}
    To analyze the multilevel scheme in~\cref{eq:scheme}, we introduce the auxiliary particles
    \begin{equation}
        \bar X_{n+1}^{\ell,j} = \bar X_n^{\ell,j} + b(\bar X_n^{\ell,j}, \mathbb E[R(\bar X_n^{\ell,j})])\Delta t_\ell + \sigma(\bar X_n^{\ell,j}, \mathbb E[R(\bar X_n^{\ell,j})]) \Delta W_n^{\ell,j}
    \end{equation}
    with $\bar X_0^{\ell,j} = X_0^{\ell,j}$ and
    \begin{equation}
        \bar{\widetilde{X}}_{n+1}^{\ell,j} = \bar{\widetilde{X}}_n^{\ell,j} + b(\bar{\widetilde{X}}_n^{\ell,j}, \mathbb E[R(\bar{\widetilde{X}}_n^{\ell,j})])\Delta t_{\ell-1} + \sigma(\bar{\widetilde{X}}_n^{\ell,j}, \mathbb E[R(\bar{\widetilde{X}}_n^{\ell,j})]) \Delta \widetilde W_n^{\ell,j}
    \end{equation}
    with $\bar{\widetilde{X}}_0^{\ell,j} = \widetilde X_0^{\ell,j}$. These are mean-field particles that do not interact and are independent of each other, except for the pairs driven by the same Brownian motions. We will use them as a stepping stone to analyze the error of the multilevel scheme, noting also that
    \begin{equation} \label{eq:law-equality}
        \bar X_n^{\ell,j} \overset{d}{=} \bar{\widetilde X}_n^{\ell+1,j}, \qquad 0\le\ell\le L-1,
    \end{equation}
    since both are identical Euler--Maruyama discretizations. We introduce ensembles $\smash{\mathbf{\bar x}_n^\ell}$, $\smash{\mathbf{\bar{\widetilde x}}_n^\ell}$, and $\smash{\mathbf{\bar x}_n^{{\mathrm{ML}},\ell}}$ analogously to the non-mean-field ensembles.

    Throughout this section, we will omit the superscript $j$ when it is not relevant. For a random vector or matrix $Y$ we write $\normZZZZ{Y}_p \coloneqq ({\mathbb{E}}[\absZZZZ{Y}^p])^{1/p}$; $\absZZZZ{\,\cdot\,}$ denotes the Euclidean norm on $\mathbb R^d$ and the Frobenius norm on matrices.\clearpage
    
    \subsection{Assumptions and auxiliary results}
    We impose Lipschitz and initial-moment assumptions on the McKean--Vlasov~SDE.\hspace{-3cm}
    \begin{assumption} \label{ass:standing}
        Throughout the analysis we fix $p \ge 2$. There exist constants $L_b$, $L_\sigma$, and $L_R$ such that, for all
        $x, x'\in\mathbb R^d$ and $y, y'\in\mathbb R^{d_R}$,
        \begin{subequations}
        \begin{align}
            \absZZZZ{b(x, y) - b(x', y')} &\le L_b\left(\absZZZZ{x - x'} + \absZZZZ{y - y'}\right), \\
            \absZZZZ{\sigma(x, y) - \sigma(x', y')} &\le L_\sigma\left(\absZZZZ{x - x'} + \absZZZZ{y - y'}\right), \\
            \absZZZZ{R(x) - R(x')} &\le L_R\,\absZZZZ{x - x'}.
        \end{align}
        \end{subequations}
        The initial distribution $\rho_0$ has a finite moment of order $p$: ${\mathbb{E}}[\absZZZZ{X_0}^p] < \infty$.
    \end{assumption}

    \Cref{ass:standing} implies that $\absZZZZ{b(x,y)} \le \absZZZZ{b(0,0)} + L_b(\absZZZZ{x} + \absZZZZ{y})$ and $\absZZZZ{\sigma(x,y)} \le \absZZZZ{\sigma(0,0)} + L_\sigma(\absZZZZ{x} + \absZZZZ{y})$; the following strong-convergence result follows from \cite{liu2024particle}.

    \begin{proposition} \label{prop:strong-conv}
        Under \cref{ass:standing}, \cref{eq:mckv} has a unique strong solution $X$ with $\sup_{t \in [0, T]} \normZZZZ{X_t}_p < \infty$. There is a constant $C_{\mathrm{sc}}$, depending only on $p$ and the problem data $b, \sigma, R, \rho_0, T$, such that for any level $\ell \ge 0$, the exact mean-field Euler--Maruyama scheme $\bar X^{\ell,j}$ for \cref{eq:mckv} with step size $\Delta t_\ell \le \Delta t_0$, driven by the initial value $X_0^{\ell,j}$ and Brownian motion $W^{\ell,j}$ coupled to $X$, satisfies
        \begin{equation} \label{eq:strong-conv}
            \left({\mathbb{E}}\left[\max_{0 \le n \le \floor{T/{\Delta t_\ell}}} \absZZZZ{\bar X_n^{\ell,j} - X_{n\Delta t_\ell}}^p\right]\right)^{1/p} \le C_{\mathrm{sc}}\, {\Delta t_\ell}^{1/2}.
        \end{equation}
        An analogous bound holds for the coarse $\bar{\widetilde{X}}^{\ell,j}$ when replacing $\Delta t_\ell$ with $\Delta t_{\ell-1}$.
    \end{proposition}

    The solution $X$ is also H\"older-$1/2$ continuous in time in $L^p$: there is a constant $C_{\mathrm H}$, depending only on $p$ and the problem data $b, \sigma, R, \rho_0, T$, such that
    \begin{equation} \label{eq:holder}
        \normZZZZ{X_t - X_s}_p \le C_{\mathrm H}\,(t - s)^{1/2}, \qquad 0 \le s \le t \le T.
    \end{equation}
    Indeed, $X$ solves a time-inhomogeneous It\^o SDE with globally Lipschitz, linearly growing coefficients and $\normZZZZ{X_0}_p < \infty$, for which \cref{eq:holder} is standard; see, e.g., \cite[Theorem 4.3]{mao2StochasticDifferential2011}.

    \begin{lemma} \label{lmm:mf-diffs-timestep}
        Under \cref{ass:standing} there exists a constant $C_\mathrm{Lem\ref*{lmm:mf-diffs-timestep}}$, depending only on $p$ and the problem data $b, \sigma, R, \rho_0, T, a$ -- in particular independent of $\ell$, $L$, and the time indices -- such that the following holds. For every level $\ell \ge 1$ and all indices $m, m' \ge 0$ with
        \begin{equation} \label{eq:mf-diffs-hyp}
            m\Delta t_\ell \le T, \qquad m'\Delta t_{\ell-1} \le T, \qquad \absZZZZ{m\Delta t_\ell - m'\Delta t_{\ell-1}} \le \Delta t_{\ell-1},
        \end{equation}
        the coupled mean-field particles satisfy
        \begin{equation} \label{eq:mf-diffs-bound}
            \normZZZZ{\bar X_m^{\ell} - \bar{\widetilde{X}}_{m'}^{\ell}}_p \le C_\mathrm{Lem\ref*{lmm:mf-diffs-timestep}}\,\Delta t_\ell^{1/2}.
        \end{equation}
        Choosing $m = n$ and $m' = \floor{n/a}$, for some $n$, recovers the coupling of a fine and corresponding coarse particle. The more general \cref{eq:mf-diffs-bound} will be needed later.
    \end{lemma}
    \begin{proof}
        Let $X$ be the exact solution of \cref{eq:mckv} driven by the common initial value $X_0 = \bar X_0^{\ell} = \bar{\widetilde{X}}_0^{\ell}$ and Brownian motion $W^{\ell}$ underlying both $\bar X^{\ell}$ and $\bar{\widetilde{X}}^{\ell}$, and set $t \coloneqq m\Delta t_\ell$ and $t' \coloneqq m'\Delta t_{\ell-1}$. Both $\bar X^{\ell}$ (step $\Delta t_\ell \le \Delta t_0$) and $\bar{\widetilde{X}}^{\ell}$ (step $\Delta t_{\ell-1} = a\Delta t_\ell \le \Delta t_0$) are exact mean-field Euler--Maruyama schemes for \cref{eq:mckv} coupled to $X$ through $(X_0, W^{\ell})$. Since \cref{eq:mf-diffs-hyp} gives $m \le \floor{T/\Delta t_\ell}$ and $m' \le \floor{T/\Delta t_{\ell-1}}$, \cref{prop:strong-conv} shows that
        \begin{equation}
            \normZZZZ{\bar X_m^{\ell} - X_t}_p \le C_{\mathrm{sc}}\,\Delta t_\ell^{1/2}, \qquad \text{and} \qquad
            \normZZZZ{\bar{\widetilde{X}}_{m'}^{\ell} - X_{t'}}_p \le C_{\mathrm{sc}}\,a^{1/2}\Delta t_\ell^{1/2}.
        \end{equation}
        By \cref{eq:holder}, $\normZZZZ{X_t - X_{t'}}_p \le C_{\mathrm H}\absZZZZ{t - t'}^{1/2} \le C_{\mathrm H}\,a^{1/2}\Delta t_\ell^{1/2}$. Therefore,
        \begin{equation}
        \begin{aligned}
            \normZZZZ{\bar X_m^{\ell} - \bar{\widetilde{X}}_{m'}^{\ell}}_p &\le \normZZZZ{\bar X_m^{\ell} - X_t}_p + \normZZZZ{X_t - X_{t'}}_p + \normZZZZ{X_{t'} - \bar{\widetilde{X}}_{m'}^{\ell}}_p\\
            &\le \left(C_{\mathrm{sc}} + (C_{\mathrm H} + C_{\mathrm{sc}})a^{1/2}\right)\Delta t_\ell^{1/2},
        \end{aligned}
        \end{equation}
        proving \cref{eq:mf-diffs-bound}.
    \end{proof}

    \subsection{Main results}
    We will prove a coupling bound in \cref{thm:main}: mean-field particles $\bar X^{\ell,j}$~and~$\bar{\widetilde{X}}^{\ell,j}$ are close to their corresponding interacting counterparts $X^{\ell,j}$ and $\widetilde{X}^{\ell,j}$. The proof of this theorem is structured as follows. \begin{itemize}
        \item We expand the error between coupled particles in terms of summed errors at previous timesteps, added to the multilevel estimation error of ${\mathbb{E}}[R(\bar X_k^\ell)]$.
        \item We split the multilevel estimation error into an error due to interaction and a statistical MLMC error.
        \item Bounding the statistical error is classical in MLMC. However, the interaction error accumulates all lower-level errors at every timestep. To control it, we use the choice of $J_\ell$ in \cref{thm:main} to ensure that the summed errors are dominated by the finest-level error.
        \item We finish by bounding the resulting recurrence relations by an ODE and applying Gr\"onwall's inequality.
    \end{itemize}

    \begin{theorem}[Pathwise coupling] \label{thm:main}
        Let \cref{ass:standing} hold and consider any $q>1$ and $a^{-1/2} \ge \epsilon > 0$. Choose
        \begin{equation} \label{eq:par-ml}
            L = \floor{2\log_a(\epsilon^{-1})} \qquad \text{and} \qquad J_\ell = \floor{C_Jq^{2(L-\ell)} a^{L-\ell}}
        \end{equation}
        with $C_J \ge 1$. Consider the error
        \begin{equation}
            E_n^\ell \coloneqq \max(e_n^\ell,\, \tilde e_n^{\ell + 1}), \qquad 0 \le \ell \le L,
        \end{equation}
        with $\tilde e_n^{L+1} \coloneqq 0$ and with
        \begin{subequations}
        \begin{align}
            e_n^\ell &\coloneqq {\mathbb{E}}\left[\max_{m\le n} \absZZZZ{X_m^{\ell} - \bar X_m^{\ell}}^p\right] \qquad \text{for $0\le \ell \le L$},\\
            \tilde e_n^\ell &\coloneqq {\mathbb{E}}\left[\max_{m\le n} \absZZZZ{\widetilde X_m^{\ell} - \bar{\widetilde{X}}_m^{\ell}}^p\right] \qquad \text{for $1 \le \ell \le L$}.
        \end{align}
        \end{subequations}
        We have, for every $0\le\ell\le L$ and $n\Delta t_\ell\le T$, that
        \begin{equation}
            (E_n^\ell)^{1/p} \le C \epsilon / q^{L-\ell},
        \end{equation}
        with $C$ independent of $L$, $n$, $\ell$, and $\epsilon$.
    \end{theorem}
    The cost of running this algorithm to a fixed time satisfies
    \begin{equation}
        \mathrm{Cost} \in \mathcal O\left(\sum_{\ell=0}^L J_\ell a^\ell\right) \subseteq \mathcal O\left(a^L q^{2L}\right) \subseteq \mathcal O\left(\epsilon^{-2 - 4\log_aq}\right).
    \end{equation}
    By choosing $q>1$ close enough to $1$, a cost of $\mathcal O(\epsilon^{-2 - \delta})$ can be achieved for any $\delta > 0$ at the cost of increasing $C$.
    \begin{corollary}[Estimation error] \label{cor:estimator}
        Let \cref{ass:standing} hold and let $q$, $\epsilon$, $L$, and $J_\ell$ be as in \cref{thm:main}. Let $\varphi\colon\mathbb R^d\to\mathbb R^{d_\varphi}$ be Lipschitz with constant $L_\varphi$ and set $N \coloneqq \floor{T/\Delta t_L}$. Then, using the multilevel estimator $E_\varphi^{\mathrm{ML}}$ defined as in \cref{eq:ml-int},
        \begin{equation}
            \normZZZZ[\big]{E_\varphi^{\mathrm{ML}}(\mathbf x_N^{{\mathrm{ML}},L}) - {\mathbb{E}}[\varphi(X_T)]}_p \le C'\epsilon
        \end{equation}
        with $C'$ independent of $L$ and $\epsilon$, at the $\mathcal O(\epsilon^{-2-4\log_aq})$ cost computed above.
    \end{corollary}
    \begin{proofof}{Proof of \cref{thm:main}}
        We can bound
        \begin{equation}
        \begin{aligned}
            e_n^\ell &\le 2^{p-1}{\mathbb{E}}\left[\max_{m\le n} \absZZZZ[\Bigg]{\sum_{k=0}^{m-1}(b(X_k^{\ell}, E_R^{\mathrm{ML}}(\mathbf x_k^{{\mathrm{ML}},\ell})) - b(\bar X_k^{\ell}, \mathbb E[R(\bar X_k^{\ell})]))\Delta t_\ell}^p\right]\\
            &\quad+ 2^{p-1}{\mathbb{E}}\left[\max_{m\le n} \absZZZZ[\Bigg]{\sum_{k=0}^{m-1}(\sigma(X_k^{\ell}, E_R^{\mathrm{ML}}(\mathbf x_k^{{\mathrm{ML}},\ell})) - \sigma(\bar X_k^{\ell}, \mathbb E[R(\bar X_k^{\ell})]))\Delta W_k^{\ell,j}}^p\right],
        \end{aligned}
        \end{equation}
        and now handle both terms separately: $e_n^\ell \le 2^{p-1} \mathrm{I} + 2^{p-1} \mathrm{II}$.

        \paragraph{Drift term.} By the triangle inequality, $\max_{m\le n}\absZZZZ{\sum_{k=0}^{m-1}c_k\Delta t_\ell} \le \sum_{k=0}^{n-1}\absZZZZ{c_k}\Delta t_\ell$; then we apply Jensen's inequality (which introduces the factor $n^{p-1}$). We continue with
        \begin{equation}
        \begin{aligned}
            \mathrm{I} &\le n^{p-1}\Delta t_\ell^p\sum_{k=0}^{n-1}{\mathbb{E}}\left[\absZZZZ{b(X_k^{\ell}, E_R^{\mathrm{ML}}(\mathbf x_k^{{\mathrm{ML}},\ell})) - b(\bar X_k^{\ell}, \mathbb E[R(\bar X_k^{\ell})])}^p\right]\\
            &\le T^{p-1}L_b^p2^{p-1}\Delta t_\ell \sum_{k=0}^{n-1}\left({\mathbb{E}}\left[\absZZZZ{X_k^{\ell} - \bar X_k^{\ell}}^p\right] + {\mathbb{E}}\left[\absZZZZ{E_R^{\mathrm{ML}}(\mathbf x_k^{{\mathrm{ML}},\ell}) - \mathbb E[R(\bar X_k^{\ell})]}^p\right]\right)\\
            &\le T^{p-1}L_b^p2^{p-1}\Delta t_\ell \sum_{k=0}^{n-1}\left(E_k^\ell + {\mathbb{E}}\left[\absZZZZ{E_R^{\mathrm{ML}}(\mathbf x_k^{{\mathrm{ML}},\ell}) - \mathbb E[R(\bar X_k^{\ell})]}^p\right]\right).
        \end{aligned}
        \end{equation}

        \paragraph{Diffusion term.} By the discrete Burkholder--Davis--Gundy inequality (with constant $C_p$) applied to the sum, which is a martingale in $m$, we obtain
        \begin{equation*}
        \begin{aligned}
            \mathrm{II} &\le C_p {\mathbb{E}}\left[\left(\sum_{k=0}^{n-1} \absZZZZ[\big]{\sigma(X_k^{\ell}, E_R^{\mathrm{ML}}(\mathbf x_k^{{\mathrm{ML}},\ell})) - \sigma(\bar X_k^{\ell}, \mathbb E[R(\bar X_k^{\ell})])}^2\,\absZZZZ{\Delta W_k^{\ell}}^2\right)^{p/2}\right]\\
            &\le C_p n^{p/2-1}\sum_{k=0}^{n-1}{\mathbb{E}}\left[\absZZZZ[\big]{\sigma(X_k^{\ell}, E_R^{\mathrm{ML}}(\mathbf x_k^{{\mathrm{ML}},\ell})) - \sigma(\bar X_k^{\ell}, \mathbb E[R(\bar X_k^{\ell})])}^p\,\absZZZZ{\Delta W_k^{\ell}}^p\right]\\
            &\le C_p C_p' T^{p/2-1} \Delta t_\ell \sum_{k=0}^{n-1}{\mathbb{E}}\left[\absZZZZ[\big]{\sigma(X_k^{\ell}, E_R^{\mathrm{ML}}(\mathbf x_k^{{\mathrm{ML}},\ell})) - \sigma(\bar X_k^{\ell}, \mathbb E[R(\bar X_k^{\ell})])}^p\right]\\
            &\le C_p C_p' T^{p/2-1} L_\sigma^p 2^{p-1} \Delta t_\ell \sum_{k=0}^{n-1}\left(E_k^\ell + {\mathbb{E}}\left[\absZZZZ[\big]{E_R^{\mathrm{ML}}(\mathbf x_k^{{\mathrm{ML}},\ell}) - \mathbb E[R(\bar X_k^{\ell})]}^p\right]\right),
        \end{aligned}
        \end{equation*}
        where the third inequality used the fact that the Gaussian increments are independent of the paths and have $p$th moments $C_p'\Delta t_\ell^{p/2}$.

        \paragraph{Multilevel interaction.} The next step is bounding the difference between the empirical multilevel and mean-field interaction terms, which occurs in the upper bounds on both the drift error and the diffusion error. We have
        \begin{equation*}
        \begin{aligned}
            &{\mathbb{E}}\left[\absZZZZ{E_R^{\mathrm{ML}}(\mathbf x_k^{{\mathrm{ML}},\ell}) - \mathbb E[R(\bar X_k^{\ell})]}^p\right]\\
            &\quad \le 2^{p-1}{\mathbb{E}}\left[\absZZZZ{E_R^{\mathrm{ML}}(\mathbf x_k^{{\mathrm{ML}},\ell}) - E_R^{\mathrm{ML}}(\mathbf{\bar x}_k^{{\mathrm{ML}},\ell})}^p\right] + 2^{p-1}{\mathbb{E}}\left[\absZZZZ{E_R^{\mathrm{ML}}(\mathbf{\bar x}_k^{{\mathrm{ML}},\ell}) - \mathbb E[R(\bar X_k^{\ell})]}^p\right]\\
            &\quad \eqqcolon 2^{p-1}(A_k^\ell)^p + 2^{p-1}(B_k^\ell)^p,
        \end{aligned}
        \end{equation*}
        where $A_k^\ell$ contains errors due to the interaction and $B_k^\ell$ contains the statistical error. We handle both terms separately, starting with the interaction error:
        \begin{align*}
            A_k^\ell &= \Big\|\sum\nolimits_{\ell'=0}^\ell \Bigl(E_R(\mathbf x_{\floor{k/{a^{\ell - \ell'}}}}^{\ell'})- E_R(\mathbf{\bar x}_{\floor{k/{a^{\ell - \ell'}}}}^{\ell'})\Bigr)\\
                    &\qquad- \sum\nolimits_{\ell'=1}^\ell\Bigl(E_R(\mathbf{\widetilde x}_{\floor{k/{a^{\ell - \ell' + 1}}}}^{\ell'}) - E_R(\mathbf{\bar{\widetilde{x}}}_{\floor{k/{a^{\ell - \ell' + 1}}}}^{\ell'})\Bigr)\Big\|_p\\
            &\le \sum\nolimits_{\ell'=0}^\ell \normZZZZ[\Big]{E_R(\mathbf x_{\floor{k/{a^{\ell - \ell'}}}}^{\ell'}) - E_R(\mathbf{\bar x}_{\floor{k/{a^{\ell - \ell'}}}}^{\ell'})}_p\\
            &\qquad+ \sum\nolimits_{\ell'=1}^{\ell}\normZZZZ[\Big]{E_R(\mathbf{\widetilde x}_{\floor{k/{a^{\ell - \ell' + 1}}}}^{\ell'}) - E_R(\mathbf{\bar{\widetilde{x}}}_{\floor{k/{a^{\ell - \ell' + 1}}}}^{\ell'})}_p\\
            &\le L_R \sum\nolimits_{\ell'=0}^\ell \normZZZZ[\Big]{X_{\floor{k/{a^{\ell - \ell'}}}}^{\ell'} - \bar X_{\floor{k/{a^{\ell - \ell'}}}}^{\ell'}}_p + L_R\sum\nolimits_{\ell'=1}^{\ell}\normZZZZ[\Big]{\widetilde X_{\floor{k/{a^{\ell - \ell' + 1}}}}^{\ell'} - \bar{\widetilde{X}}_{\floor{k/{a^{\ell - \ell' + 1}}}}^{\ell'}}_p\\
            &\le L_R \sum\nolimits_{\ell'=0}^\ell \Bigl(e_{\floor{k/{a^{\ell - \ell'}}}}^{\ell'}\Bigr)^{1/p} + L_R\sum\nolimits_{\ell'=1}^{\ell}\Bigl(\tilde e_{\floor{k/{a^{\ell - \ell' + 1}}}}^{\ell'}\Bigr)^{1/p}\\
            &\le L_R \left(E^\ell_k\right)^{1/p} + 2 L_R \sum\nolimits_{\ell'=0}^{\ell-1} \Bigl(E_{\floor{k/{a^{\ell - \ell'}}}}^{\ell'}\Bigr)^{1/p}.
        \end{align*}

        Now we bound the statistical error $B_k^\ell$. The level-$0$ interaction has no coarse partner, so we separate it from the sum over $\ell'\ge1$; using \cref{eq:law-equality}, the mean-field expectations telescope to ${\mathbb{E}}[R(\bar X_k^\ell)]$, and
        \begin{equation*}
        \begin{aligned}
            B_k^\ell &= \bigg\|\left(E_R(\mathbf{\bar x}_{\floor{k/{a^\ell}}}^0) - {\mathbb{E}}[R(\bar X_{\floor{k/{a^\ell}}}^0)]\right)\\
            &\qquad+ \sum_{\ell'=1}^\ell\Big(E_R(\mathbf{\bar x}_{\floor{k/{a^{\ell - \ell'}}}}^{\ell'}) - E_R(\mathbf{\bar{\widetilde x}}_{\floor{k/{a^{\ell - \ell' + 1}}}}^{\ell'})\\
            &\qquad\qquad\quad- {\mathbb{E}}[R(\bar X_{\floor{k/{a^{\ell-\ell'}}}}^{\ell'})] + {\mathbb{E}}[R(\bar{\widetilde{X}}_{\floor{k/{a^{\ell-\ell'+1}}}}^{\ell'})]\Big)\bigg\|_p\\
            &\le \normZZZZ[\Big]{E_R(\mathbf{\bar x}_{\floor{k/{a^\ell}}}^0) - {\mathbb{E}}[R(\bar X_{\floor{k/{a^\ell}}}^0)]}_p\\
            &\qquad+ \sum_{\ell'=1}^\ell \Big\|E_R(\mathbf{\bar x}_{\floor{k/{a^{\ell - \ell'}}}}^{\ell'}) - E_R(\mathbf{\bar{\widetilde x}}_{\floor{k/{a^{\ell - \ell' + 1}}}}^{\ell'})\\
            &\qquad\qquad\quad- \Bigl({\mathbb{E}}[R(\bar X_{\floor{k/{a^{\ell-\ell'}}}}^{\ell'})] - {\mathbb{E}}[R(\bar{\widetilde{X}}_{\floor{k/{a^{\ell-\ell'+1}}}}^{\ell'})]\Bigr)\Big\|_p\\
            &\le C_\mathrm{MZ} J_0^{-1/2}\normZZZZ{R(\bar X_{\floor{k/{a^\ell}}}^0) - {\mathbb{E}}[R(\bar X_{\floor{k/{a^\ell}}}^0)]}_p\\
            &\qquad+ C_\mathrm{MZ}\sum_{\ell'=1}^\ell J_{\ell'}^{-1/2}\normZZZZ{R(\bar X_{\floor{k/{a^{\ell - \ell'}}}}^{\ell'}) - R(\bar{\widetilde{X}}_{\floor{k/{a^{\ell - \ell' + 1}}}}^{\ell'})}_p\\
            &\le 2 C_\mathrm{MZ} M_R J_0^{-1/2} + C_\mathrm{MZ} L_R \sum_{\ell'=1}^\ell J_{\ell'}^{-1/2}\normZZZZ{\bar X_{\floor{k/{a^{\ell - \ell'}}}}^{\ell'} - \bar{\widetilde{X}}_{\floor{k/{a^{\ell - \ell' + 1}}}}^{\ell'}}_p\\
            &\le 2 C_\mathrm{MZ} M_R J_0^{-1/2} + C_\mathrm{MZ} L_R C_\mathrm{Lem\ref*{lmm:mf-diffs-timestep}}\Delta t_0^{1/2} \sum_{\ell'=1}^\ell J_{\ell'}^{-1/2} a^{-\ell'/2}\\
            &\le C_B \sum\nolimits_{\ell'=0}^\ell J_{\ell'}^{-1/2} a^{-\ell'/2},
        \end{aligned}
        \end{equation*}
        where we have defined the two constants $M_R \coloneqq \sup_{n\Delta t_0 \le T}\normZZZZ{R(\bar X_n^0)}_p$ and $C_B \coloneqq \max\{2 C_\mathrm{MZ} M_R,\, C_\mathrm{MZ} L_R C_\mathrm{Lem\ref*{lmm:mf-diffs-timestep}}\Delta t_0^{1/2}\}$. The third line applies the Marcinkiewicz--Zygmund inequality with constant $C_\mathrm{MZ}$; the fourth uses the Lipschitz continuity of $R$ together with $\normZZZZ{R(\bar X_{\floor{k/{a^\ell}}}^0) - {\mathbb{E}}[R(\bar X_{\floor{k/{a^\ell}}}^0)]}_p \le 2 M_R$, and the fifth uses \cref{lmm:mf-diffs-timestep}. The latter applies since the condition \cref{eq:mf-diffs-hyp} is satisfied: both times lie in $(k\Delta t_\ell - \Delta t_{\ell'-1}, k\Delta t_\ell]$. Finiteness of $M_R$ follows from $\normZZZZ{R(\bar X_n^0)}_p \le \absZZZZ{R(0)} + L_R\normZZZZ{\bar X_n^0}_p$ and the moment bound of \cref{prop:strong-conv}.

        \paragraph{Similarly bounding $\tilde e_n^{\ell+1}$.} By analogous arguments, one can derive the same upper bound for $\tilde e_n^{\ell+1}$ as for $e_n^\ell$. The dynamics of the level-$(\ell+1)$ coarse particles differs from that of the level-$\ell$ fine particles only by not interacting with themselves. However, self-interaction does not influence the derivations. The state and interaction arguments of $b$ and $\sigma$ are isolated by the Lipschitz bound, leading to the same upper bounds on $\mathrm{I}$ and $\mathrm{II}$. As a result, the bound on $e_n^\ell$ also~holds~for~$E_n^\ell$.

        \paragraph{Bringing it together.} The numbers $K$ through $K''''$ will represent constants independent of $n$, $\ell$, $L$, and $\epsilon$. We compute
        \begin{equation*}
        \begin{aligned}
            E_n^\ell &\le K\Delta t_\ell\sum_{k=0}^{n-1}\left[E_k^\ell + (A_k^\ell)^p + (B_k^\ell)^p\right]\\
            &\le K'\Delta t_\ell\sum_{k=0}^{n-1}\left[E_k^\ell + \left(\sum_{\ell'=0}^{\ell-1} \left(E^{\ell'}_{\floor{k/{a^{\ell - \ell'}}}}\right)^{1/p}\right)^p + \left(\sum_{\ell'=0}^\ell J_{\ell'}^{-1/2}a^{-\ell'/2}\right)^p\right]\\
            &\le K''\Delta t_\ell\sum_{k=0}^{n-1}\left[E_k^\ell + \left(\sum_{\ell'=0}^{\ell-1} \left(E^{\ell'}_{\floor{k/{a^{\ell - \ell'}}}}\right)^{1/p}\right)^p + \left(\frac{q - q^{-\ell}}{q-1}\,\frac1{q^{L-\ell}a^{L/2}}\right)^p\right],
        \end{aligned}
        \end{equation*}
        where $K''$ absorbs a constant related to the floor function used in defining $J_{\ell'}$. We call the last term, the statistical error injected every timestep, $F_\ell^p$; its geometric decay with $q^{-(L-\ell)}$ is the vital element that will allow us to bound $E_n^\ell$. We define
        \begin{equation} \label{eq:def-P}
            E_n^\ell \le P_n^\ell \coloneqq K''\Delta t_\ell\sum_{k=0}^{n-1}\Biggl[P_k^\ell + \Biggl(\sum_{\ell'=0}^{\ell-1} \biggl(P^{\ell'}_{\floor{k/{a^{\ell - \ell'}}}}\biggr)^{1/p}\Biggr)^p + F_\ell^p\Biggr].
        \end{equation}
        We note that, since the elements in the outer sum of $P_n^\ell$ are increasing, $P_n^\ell \le P^\ell(n\Delta t_\ell)$ with
        \begin{equation} \label{eq:def-P-cont}
            P^\ell(t) = K''\int_0^t \Biggl[P^\ell(s) + \Biggl(\sum_{\ell'=0}^{\ell-1} \left(P^{\ell'}(s)\right)^{1/p}\Biggr)^p + F_\ell^p\Biggr]\,\mathrm ds
        \end{equation}
        with solution
        \begin{equation}
            P^\ell(t) = K''\int_0^t \exp(K''(t-s)) \Biggl[\Biggl(\sum_{\ell'=0}^{\ell-1}(P^{\ell'}(s))^{1/p}\Biggr)^p + F_\ell^p\Biggr]\mathrm ds.
        \end{equation}
        We will now prove by induction over $\ell$ that $P^\ell(t) \ge (F_\ell/F_{\ell-1})^p P^{\ell-1}(t)$ for all $t$ and $\ell$. For $\ell=1$, this is clear. For $\ell>1$, we use the induction hypothesis and the fact that $F_\ell/F_{\ell-1}$ is decreasing in $\ell$ (as is easily verified) to bound
        \begin{align*}
            &\left(\sum_{\ell'=0}^{\ell-1}(P^{\ell'}(s))^{1/p}\right)^p + F_\ell^p = \left((P^0(s))^{1/p} + \sum_{\ell'=0}^{\ell-2} (P^{\ell'+1}(s))^{1/p}\right)^p + F_\ell^p\\
            &\qquad\ge \left(\sum_{\ell'=0}^{\ell-2} \frac{F_{\ell'+1}}{F_{\ell'}}(P^{\ell'}(s))^{1/p}\right)^p + F_\ell^p \ge\frac{F_\ell^p}{F_{\ell-1}^p}\left[\left(\sum_{\ell'=0}^{\ell-2}(P^{\ell'}(s))^{1/p}\right)^p + F_{\ell-1}^p\right].
        \end{align*}
        This proves the desired inequality and, more broadly, $P^{\ell_1}(t) \ge (F_{\ell_1}/F_{\ell_2})^p P^{\ell_2}(t)$ for all $\ell_1\ge\ell_2$ and $t$.

        With this inequality, we can now bound $P^\ell$ proper. We bound \cref{eq:def-P-cont} using the inequality we just derived:
        \begin{equation}
            P^\ell(t) \le K''\int_0^t \biggl[\biggl(1+\biggl(\sum_{\ell'=0}^{\ell-1}\frac{F_{\ell'}}{{F_\ell}}\biggr)^p\biggr)P^\ell(s) + F_\ell^p\biggr]\mathrm ds.
        \end{equation}
        This sum of $F$ fractions is bounded by $\frac1{q-1}$. By the Gr\"onwall inequality, we conclude that
        \begin{equation}
            P^\ell(t) \le K'' F_\ell^p \int_0^t \exp\biggl(K''\left(1 + \frac1{(q-1)^p}\right) (t - s)\biggr)\mathrm ds \le K''' F_\ell^p.
        \end{equation}
        Since $E_n^\ell \le P_n^\ell \le P^\ell(n\Delta t_\ell) \le K''' F_\ell^p$, we conclude that (implicitly defining $C$)
        \begin{equation}
            (E_n^\ell)^{1/p} \le K'''' F_\ell = K'''' \frac{q - q^{-\ell}}{q-1}\,\frac1{q^{L-\ell}a^{L/2}} \le C \epsilon / q^{L - \ell},
        \end{equation}
        proving the theorem.
    \end{proofof}
    \begin{proofof}{Proof of \cref{cor:estimator}}
        This proof follows the structure of the multilevel interaction bound in the proof of \cref{thm:main}. By the triangle inequality,
        \begin{equation} \label{eq:cor-split}
        \begin{aligned}
            &\normZZZZ[\big]{E_\varphi^{\mathrm{ML}}(\mathbf x_N^{{\mathrm{ML}},L}) - {\mathbb{E}}[\varphi(X_T)]}_p \le \underbrace{\normZZZZ[\big]{E_\varphi^{\mathrm{ML}}(\mathbf x_N^{{\mathrm{ML}},L}) - E_\varphi^{\mathrm{ML}}(\mathbf{\bar x}_N^{{\mathrm{ML}},L})}_p}_{\text{(i) interaction}}\\
            &\qquad + \underbrace{\normZZZZ[\big]{E_\varphi^{\mathrm{ML}}(\mathbf{\bar x}_N^{{\mathrm{ML}},L}) - {\mathbb{E}}[\varphi(\bar X_N^{L})]}_p}_{\text{(ii) statistical}} + \underbrace{\absZZZZ[\big]{{\mathbb{E}}[\varphi(\bar X_N^{L})] - {\mathbb{E}}[\varphi(X_T)]}}_{\text{(iii) bias}}.
        \end{aligned}
        \end{equation}
        We bound the three terms separately.

        \paragraph{Term (i).} We can repeat the bound on $A_k^\ell$ with $(k, \ell) \leftarrow (N, L)$, and bounding $E_\varphi^{\mathrm{ML}}$ instead of $E_R^{\mathrm{ML}}$:
        \begin{equation}
            \normZZZZ[\big]{E_\varphi^{\mathrm{ML}}(\mathbf x_N^{{\mathrm{ML}},L}) - E_\varphi^{\mathrm{ML}}(\mathbf{\bar x}_N^{{\mathrm{ML}},L})}_p
            \le L_\varphi \left(E^L_N\right)^{1/p} + 2 L_\varphi \sum\nolimits_{\ell'=0}^{L-1} \Bigl(E_{\floor{N/{a^{L - \ell'}}}}^{\ell'}\Bigr)^{1/p}.
        \end{equation}
        \Cref{thm:main} then applies to each term, yielding
        \begin{equation}
            \text{(i)} \le L_\varphi C\epsilon\Bigl(1 + 2\sum\nolimits_{\ell'=0}^{L-1} q^{\ell'-L}\Bigr) \le L_\varphi C\left(1 + \frac2{q-1}\right)\epsilon.
        \end{equation}

        \paragraph{Term (ii).} This is the bound on $B_k^\ell$ with $(k, \ell) \leftarrow (N, L)$ and computing expectations of $\varphi$ instead of $R$. Hence, with $C_{B,\varphi} \coloneqq \max\{2C_\mathrm{MZ}M_\varphi,\, C_\mathrm{MZ}L_\varphi C_\mathrm{Lem\ref*{lmm:mf-diffs-timestep}}\Delta t_0^{1/2}\}$ where $M_\varphi \coloneqq \sup_{n\Delta t_0 \le T}\normZZZZ{\varphi(\bar X_n^0)}_p$,
        \begin{equation}
            \text{(ii)} \le C_{B,\varphi}\sum_{\ell'=0}^L J_{\ell'}^{-1/2}a^{-\ell'/2}
            \le \sqrt2\, C_{B,\varphi}\, \frac{q - q^{-L}}{q-1}\, a^{-L/2}
            \le \sqrt2\, C_{B,\varphi}\, \frac{q}{q-1}\, a^{1/2}\,\epsilon,
        \end{equation}
        where the last step used $a^{-L/2} \le a^{1/2}\epsilon$, since $L \ge 2\log_a(\epsilon^{-1}) - 1$.

        \paragraph{Term (iii).} Since $T - N\Delta t_L < \Delta t_L$, combining the Lipschitz continuity of $\varphi$, \cref{prop:strong-conv}, and \cref{eq:holder} gives
        \begin{equation}
        \begin{aligned}
            \text{(iii)} &\le L_\varphi {\mathbb{E}}[\absZZZZ[\big]{\bar X_N^{L} - X_T}] \le L_\varphi\left(\normZZZZ[\big]{\bar X_N^{L} - X_{N\Delta t_L}}_p + \normZZZZ[\big]{X_{N\Delta t_L} - X_T}_p\right)\\
            &\le L_\varphi\left(C_\mathrm{sc} + C_\mathrm{H}\right)\Delta t_L^{1/2} \le L_\varphi\left(C_\mathrm{sc} + C_\mathrm{H}\right)(a\Delta t_0)^{1/2}\,\epsilon.
        \end{aligned}
        \end{equation}
        Summing the three bounds proves the claim with
        \begin{equation}
            C' \coloneqq L_\varphi C\left(1 + \frac2{q-1}\right) + \sqrt2\,C_{B,\varphi}\frac{q}{q-1}a^{1/2} + L_\varphi(C_\mathrm{sc}+C_\mathrm{H})(a\Delta t_0)^{1/2}.
        \end{equation}
    \end{proofof}

\section{Numerical illustration} \label{sec:num}
    We briefly illustrate how the practical performance of the multilevel algorithm compares to the derived bounds. Consider the McKean--Vlasov SDE
    \begin{equation}
        \mathrm dX_t = 2\,(\cos({\mathbb{E}}[X_t]) - X_t)\,\mathrm dt + 1.05\,X_t\,\mathrm dW_t, \qquad X_0 \sim \mathcal N(1,0.25),
    \end{equation}
    where $\mathcal N(\mu, \sigma^2)$ denotes the normal distribution with mean $\mu$ and variance $\sigma^2$. We set $T = 2$, $a = 2$, and $\Delta t_0 = 0.5$; furthermore, as observable we use the function
    \begin{equation}
        \varphi(x) = \sin(7x).
    \end{equation}
    We use the multilevel algorithm with ensemble sizes that follow \cref{eq:par-ml}, fixing $C_J = 32$ and varying $q \in \{1,\,1.1,\,1.3,\,1.5\}$. For $q>1$, we will compare our errors to the $\mathcal O(\mathrm{Cost}^{-1/(2+4\log_aq)})$ upper bound from \cref{cor:estimator}. We note that the choice $q=1$ is not covered by that result, and that the constant of our upper bound quickly blows up to infinity as $q\to1$.

    We will compute a reference solution with a single-level simulation that uses level $L=15$ and an ensemble of $J = 2\cdot10^6$ samples. Then, for different pairs $(q, L)$, we run $400$ independent multilevel simulations and compute the root mean square error (RMSE) of $E_\varphi^{\mathrm{ML}}(\mathbf x_N^{{\mathrm{ML}},L})$ compared to the reference value. This RMSE is plotted against a relative cost measure,
    \begin{equation}
        J_0\floor{\frac T{\Delta t_0}} + \sum_{\ell=1}^L J_\ell\,\Bigl(\floor{\frac T{\Delta t_\ell}} + \floor{\frac T{\Delta t_{\ell-1}}}\Bigr),
    \end{equation}
    which accounts for updating every particle at every timestep, and evaluating $R$.

    The results are shown in \cref{fig:results}, together with the convergence rates of the upper bounds. We make three main observations. \begin{itemize}
        \item The derived rates for $q>1$ are closely tracked by the empirical errors. As a result, smaller $q$ values yield better convergence rates than larger ones. The smaller $q$ is chosen, the longer it takes until the asymptotic~regime~is~reached.
        \item Despite the fact that $q=1$ is not covered by the theory, it yields the best cost-error results in practice, narrowly edging out $q=1.1$.
        \item Whereas the constant factor in the upper bound blows up as $q\to1$, the empirical results show similar errors for all $q$ values at low cost. The trade-off between asymptotics and a constant prefactor that the theory permits is absent, making small $q$ values the clear best choice in this example.
    \end{itemize}

    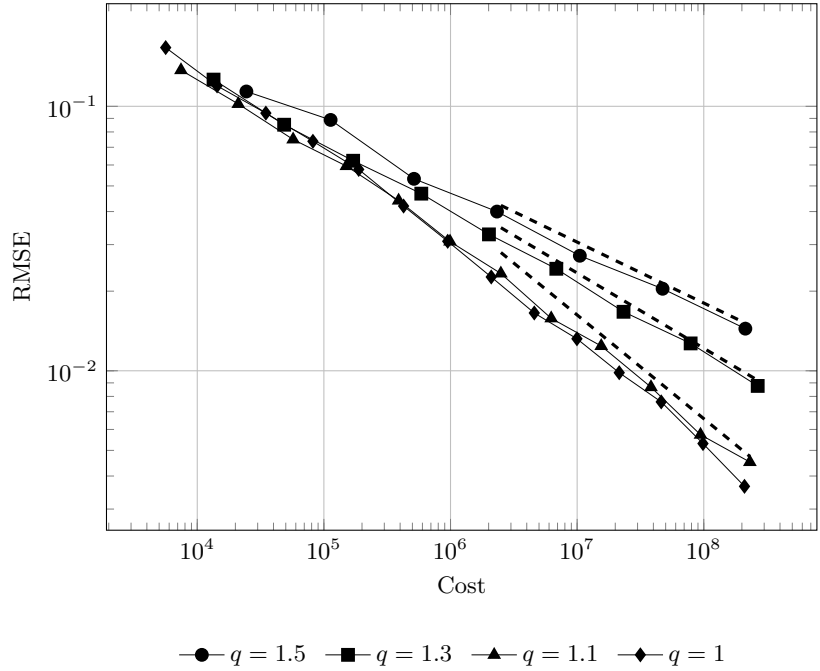
\begin{figure}[t]
        \centering
        \pgfplotstableread[col sep=comma]{data/q=1.5.csv}\tableqA
        \pgfplotstableread[col sep=comma]{data/q=1.3.csv}\tableqB
        \pgfplotstableread[col sep=comma]{data/q=1.1.csv}\tableqC
        \begin{tikzpicture}
        \begin{groupplot}[group style={group size=1 by 1, horizontal sep=2cm, vertical sep=2cm},
        width=.9\textwidth, height=0.7\textwidth,
        legend style={transpose legend, legend columns=0, draw=none}]

        \nextgroupplot[name=mainplot, xmode=log, ymode=log, grid=major,
        xlabel={$\mathrm{Cost}$}, ylabel={$\mathrm{RMSE}$},
        legend to name=grouplegend, unbounded coords=discard,
        every axis plot/.append style={mark options={solid, scale=1.1, thick}}]

        \addplot[color=black, mark=*, very thin] table [x=cost, y=RMSE] {\tableqA};
        \addlegendentry{$q=1.5\;\;$}
        \getlastpoint{\tableqA}
        \pgfmathsetmacro{\slope}{-1 / (2 + 4 * log2(1.5))}
        \setstartX
        \addplot [color=black, dashed, very thick, forget plot, samples=2,
                domain=\startX:\lastX] {1.05*\lastY * (x/\lastX)^(\slope)};

        \addplot[color=black, mark=square*, very thin] table [x=cost, y=RMSE] {\tableqB};
        \addlegendentry{$q=1.3\;\;$}
        \getlastpoint{\tableqB}
        \pgfmathsetmacro{\slope}{-1 / (2 + 4 * log2(1.3))}
        \setstartX
        \addplot [color=black, dashed, very thick, forget plot, samples=2,
                domain=\startX:\lastX] {1.05*\lastY * (x/\lastX)^(\slope)};

        \addplot[color=black, mark=triangle*, very thin] table [x=cost, y=RMSE] {\tableqC};
        \addlegendentry{$q=1.1\;\;$}
        \getlastpoint{\tableqC}
        \pgfmathsetmacro{\slope}{-1 / (2 + 4 * log2(1.1))}
        \setstartX
        \addplot [color=black, dashed, very thick, forget plot, samples=2,
                domain=\startX:\lastX] {1.05*\lastY * (x/\lastX)^(\slope)};

        \addplot[color=black, mark=diamond*, very thin] table [x=cost, y=RMSE, col sep=comma]{data/q=1.0.csv};
        \addlegendentry{$q=1\;\;$}

        \end{groupplot}

        \node at (mainplot.south) [below, yshift=-3\pgfkeysvalueof{/pgfplots/every axis title shift}] {\ref*{grouplegend}};
        \end{tikzpicture}
        \vspace{-.3cm}
        \caption{Comparison of the empirical RMSE of the multilevel simulation algorithm for different $q$ values, with the dashed lines representing the convergence rates of the upper bounds derived in \cref{cor:estimator} for $q>1$.}
        \vspace{-.3cm}
        \label{fig:results}
    \end{figure}

    The code to reproduce this experiment is available at the Zenodo repository \url{https://doi.org/10.5281/zenodo.22098916}.

\section{Discussion} \label{sec:disc}
    We have proven a convergence result for a single-ensemble multilevel Monte Carlo algorithm for McKean--Vlasov SDEs, assuming interaction of the form ${\mathbb{E}}[R(X_t)]$ and globally Lipschitz coefficients and $R$. Since particles on different levels are not independent, a coupling argument must resort to bounding the multilevel estimation error with the triangle inequality (as in the bound on $A_k^\ell$ in our proof). A resulting $L$-dependent bound would accumulate in the Gr\"onwall inequality and make it infeasible to control the algorithm's overall error.
    
    We solved this issue by requiring geometrically decreasing errors $\sim q^{-(L-\ell)}$ on coarser levels $\ell$ with $q>1$, such that their sum over all levels is $L$-independent. This is a flexible approach that may also be useful for the methods in, e.g., \cite{bouillonSingleEnsembleMultilevelMonte2026a,chernov2021,hoelMultilevelEnsembleKalman2016} (after tweaking those schemes such that particles only use lower-level particles to compute the interaction term). The theoretical bound suggests a clear trade-off: the closer $q$ is to $1$, the better the asymptotic convergence rate but the larger the constant factor in the bound.

    Both the exclusion of $q=1$ from any theory and the exploding constant for $q\to1$ result from the application of the triangle inequality, which is more than likely not sharp. A simple numerical experiment is consistent with this interpretation, showing that lowering $q$ up to and including $1$ improves the performance in practice. Hence, while the theory is valuable by proving rates arbitrarily close to $\mathrm{Cost} \in \mathcal O(\epsilon^{-2})$, it does not yet fully capture the behavior of the algorithm.

    Further research is required to decrease the gap between theory and practice for $q\to1$. Equally valuable would be extensions of the methodology to coefficients that are not (globally) Lipschitz or to more general interaction terms, settings in which multiple-ensemble MLMC algorithms are often formulated. Lastly, a detailed performance analysis between different MLMC paradigms and algorithms for McKean--Vlasov equations is of clear interest.

\section*{Acknowledgments}
    We are grateful to Thijs Steel for his feedback, which improved the presentation of this manuscript. This work was financed by the Fonds Wetenschappelijk Onderzoek -- Vlaanderen (FWO) under grant 1169725N.

\bibliographystyle{spmpsci}
\bibliography{references}

\end{document}